\documentclass[oneside,reqno]{amsart}

\usepackage{Macros}

\title{Generalized Logarithmic Sobolev Inequality by the JKO Scheme}
\author{Thibault Caillet, Fanch Coudreuse}

\address{Universite Sorbonne Paris Nord, IUT de Saint-Denis, Place du 8 Mai 1945, 93200 Saint-Denis, France.}

\email{thibault.caillet@univ-paris13.fr}

\address{Universite Claude Bernard Lyon 1, ICJ UMR5208, CNRS, Ecole Centrale de Lyon, INSA Lyon, Université Jean Monnet, 69622 Villeurbanne, France.}

\email{coudreuse@math.univ-lyon1.fr}

\makeindex

\begin{document}

\begin{abstract}
    Using a discrete Bakry-Émery method based on the JKO scheme, relying on the dissipation of entropy and Fisher information along a discrete flow, we establish new generalized logarithmic Sobolev inequality for log-concave measures of the form $e^{-V}$ under strict convexity assumptions on $V$. We then show how this method recovers some well-known inequalities. This approach can be viewed as interpolating between the Bakry-Émery method and optimal transport techniques based on geodesic convexity.
\end{abstract}

\maketitle

\section{Introduction}

    One says that a probability measure $\eta$ defined on a domain $\Omega$ satisfies a generalized logarithmic Sobolev inequality if, for all sufficiently regular functions $g$ with $\int_\Omega g \dd{\eta} = 1$, one has
    \begin{equation} \label{eq: Generalized_Sobolev_Ineq}
        \int g \log g \dd{\eta} \leq \int_{\mathbb{R}^d} G(\nabla \log g) g \dd{\eta} 
    \end{equation}
    for some function $G : \bb{R}^d \to [0,+\infty)$. Proving such a general inequality is generally quite challenging, even in the simplest cases. Such inequality has been mostly studied for its connections with other theories, in particular when $G$ is quadratic. Among the many approaches in the literature, two main frameworks emerge as general tools for studying such inequalities. The first is the Bakry--Émery approach based on Markov flows, which applies in the rather general setting of Markov semi-groups. The second approach is based on optimal transport and the geodesic convexity of functionals on the space of probability measures. In fact, these two approaches are often connected through the theory of gradient flows in Wasserstein space, and results can frequently be translated from one framework to the other.
    
    In this work, we propose to combine these two approaches and prove several Log-Sobolev-type inequalities using a discrete version of the Bakry-Émery technique. It is well known since the seminal work of Jordan, Kinderlehrer and Otto \cite{JKO} that the so-called JKO scheme can be used to approximate various classes of doubly non-linear diffusion equations, notably some that arise in the study of functional inequalities via the Bakry-Émery method. Our strategy is to derive two dissipation identities along the discrete flow: an energy dissipation inequality obtained via geodesic convexity, and a dissipation inequality for a generalized Fisher information using the so-called five gradients inequality. By combining these estimates with the convergence of the discrete flow to equilibrium, we obtain logarithmic Sobolev inequalities under suitable compatibility conditions between the energy and the generalized Fisher information. 
    
    We believe that this discrete analogue not only interpolates between the Bakry–Émery and optimal transport methods, but may also be used to derive inequalities under conditions that do not fall within the scope of either approach.

    Although we were unable to identify specific examples yielding genuinely new inequalities beyond those already known in the functional inequalities literature, we were able to recover several classical inequalities in the case of potentials 
    $V$ with $p$-power moduli of convexity. Furthermore, in the radial case, our method yields a simple criterion (Corollary~\ref{coro: radially_sym_case}).

\subsection{Statement of the Main Result}
    
    Before stating our theorem, we shall introduce the following notations that will remain consistent in the rest of the paper.
    \begin{itemize}
        \item $h$ is strictly convex, superlinear, radially symmetric, of class $C^1(\bb{R}^d)$, and satisfies $h(0) = 0$. We shall call a function $h$ satisfying these assumptions a cost function.
        \item $H : \bb{R}^d \to \bb{R}$ is convex, radially symmetric, of class $C^1(\bb{R}^d \setminus \{ 0 \})$ and satisfies $H(0) = 0$.
        \item $L : \bb{R}^d \to \bb{R}$ is convex, superlinear function and satisfies $L(0) = 0$.
        \item $\Omega$ is a convex domain with non-empty interior, it will be, otherwise stated, bounded.
        \item $V : \Omega \to \bb{R}$ is a $C^1(\Omega) \cap W^{1,\infty}(\Omega)$ convex function such that $\eta := e^{-V} \in \scr{P}(\Omega)$.
    \end{itemize}

    With these notations, our main result is the following.

    \begin{theorem} \label{thm: log_Sobolev}
        Assume that $V$ admits moduli of convexity $\sigma$ and modulus of monotonicity $\omega$ (see definition \ref{def: modulus}) on $\Omega$ eventually unbounded. \\
        Under the point-wise bound for all $z \neq 0$:
        \begin{equation}
            L(z) + L^*(\nabla h^*(z)) \leq \sigma(\nabla h^*(z)) + \alpha(z) \omega(\nabla h^*(z)) \qquad \alpha(z) := \frac{|\nabla H(z)|}{|\nabla h^*(z)|}
        \end{equation}
        Then, for $G := H + L$, $\eta$ satisfies the modified Log-Sobolev inequality: for all $g > 0$ with $\int_\Omega g \dd{\eta} = 1$, and $\nabla \log g \in W^{1,\infty}(\Omega)$ it holds that
        \begin{equation}
            \int_\Omega g \log g \dd{\eta} \leq \int_\Omega G(\nabla \log g) g \dd{\eta} 
        \end{equation}
    \end{theorem}

    If the conditions stated in the theorem appear complicated, we shall show in section $4.1$ how to derive a simpler criterion than in the theorem.

\subsection{A Remark on the choice of Entropy and on regularity hypotheses}

    It is natural to ask if the method applies to more complicated inequalities, replacing the log-entropy by more general $\phi$-entropies, in the spirit of \cite{IneqGeneEntOptTranspo} and \cite{EntFlowFuncIneqConvSets}. It will be clear by looking at the proof that we need to focus on the logarithm in the general case. Indeed, whenever we work with a generic $H$, one needs to work with the classical log-entropy, as doubly non-linear versions of the five-gradient inequalities are not true in general. On the other hand, if one chooses $H = C h^*$, then we can use the geodesic convexity of the entropy under McCann conditions (see \cite{McCannCondition}) to obtain a more general result. As justifying the computations in this general case proves to be trickier and applications of more complicated versions of Log-Sobolev inequalities are less clear, we chose to focus on the logarithmic entropy.

    Similarly, regularity hypotheses on $H,L,V,h$ might be dropped by an approximation procedure; however we chose to keep this less general framework as it makes the proof much easier without losing the heart of the argument. Since our goal was more an emphasis on the method than on obtaining the most general condition, we choose to stick with the hypothesis.
    
\subsection{Comparison with the Literature}

    The logarithmic Sobolev inequality is usually stated in the form  
    \begin{equation*}
        \int f^2 \log f^2 \dd{\eta} \leq C \int |\nabla f|^2 \dd{\eta} \qquad \int f^2 \dd{\eta} = 1
    \end{equation*}
    which corresponds to the case $G = C |\cdot|^2$. This inequality was originally introduced by Gross (see \cite{Gross1}) and proven to hold in the case of the standard Gaussian measure on $\mathbb{R}^d$ with the optimal constant $1/2$. Since then, many probability measures satisfying such inequality for different $C$ have been found. A major characterization is the so-called Bakry-Émery criterion, based on the carré du champ theory and linked to the curvature-dimension criterion. This can be used, for instance, in the case $\mu \propto e^{-V}$ with $V$ $\Lambda$-convex to obtain the constant $1/(2\Lambda)$. We refer to \cite{Ledoux1}, \cite{Bakry1}, and the monograph \cite{LivreIvan} for this part of the theory. This theory extends well beyond the standard $\mathbb{R}^d$ space, covering Riemannian manifolds, metric spaces, and even quite general Markov processes without referring to a particular underlying smooth structure. For further references, we refer again to the monograph \cite{LivreIvan}, as well as the lecture notes \cite{LecNoteLogSoboIneq} and the book \cite{InitLogSoboIneq}.  
    
    A straightforward generalization replaces the square by with $p$-power, yielding inequalities of the form  
    \begin{equation*}
        \int f^q \log f^q \dd{\eta} \leq C \int |\nabla f|^q \dd{\eta} \qquad \int f^q \dd{\eta} = 1
    \end{equation*}
    The class of measures satisfying such inequalities has been studied in great depth by Bobkov and Zegarlinski in \cite{EntBoundIsop} and Bobkov and Ledoux in \cite{BobkovLedoux}, where links to other inequalities, concentration phenomena, and examples have been found (see also \cite{GeneOptLpEucLogSoboIneqHamJacEq} for a Euclidean version).   
    
    Another popular choice of modified Log-Sobolev inequalities are inequalities of the form  
    \begin{equation*}
        \int e^g \log e^g \dd{\eta} \leq C \int H_{a,q}(\nabla g) e^g \dd{\eta}
    \end{equation*}
    where  
    \begin{equation*}
        H_{a,q} := 
        \begin{cases} 
            \frac{1}{2} |x|^2 & \text{if } |x| \leq a \\ 
            \frac{a^{2-q}}{q} |x|^2 + \frac{1}{2} - \frac{1}{q} & \text{if } |x| \geq a.
        \end{cases}
    \end{equation*}
    This type of inequality is a slight improvement over the $q$-logarithmic Sobolev inequality for $q \in [1,2)$ since the behavior is better near zero. It is also the type of inequality needed to recover the Poincaré inequality for $\mu$ by linearization around $1$ of the logarithmic Sobolev inequality. In the case $q = +\infty$, this inequality has been proven to hold for the exponential measure $\eta \propto e^{-|x|}$ by Bobkov and Ledoux \cite{PoinIneqTalagConcPhenoExpoDitrib}. The general class of inequalities was introduced by Gentil, Guillin, and Miclo in \cite{GentilGuillinMiclo1}, with extensions to log-concave measures with control on tail behavior in \cite{ModLogSoboIneqNullCurv}. For instance, the exponential measure $\eta \propto e^{-|x|^p}$ is known to satisfy such an inequality on $\bb{R}$. Since then, such inequalities have been deeply studied on the real line, and sufficient conditions beyond the log-concave case have been found by Barthe and Roberto in \cite{ModLogSoboIneqR}. Unfortunately, since those conditions do not depend on any modulus of convexity for $V$, these results cannot be obtained using our theorem.  
    
    As explained in the introduction, using optimal transport to derive logarithmic Sobolev inequalities is not new. For instance, for the classical logarithmic Sobolev inequality, as shown by Otto and Villani \cite{GeneIneqTalaLinkLogSoboIneq}, or by Cordero \cite{AppMassTransGaussTypeIneq}. This can also be linked to the curvature-dimension condition and serves as a building block for one of the modern perspectives on these conditions; see Villani’s monographs \cite{VillaniO&N}, \cite{VillaniOT} in this regard.  
    
    This type of reasoning was widely expanded for the non-quadratic case using a modified version of geodesic convexity, applied by Agueh \cite{Agueh} in his proof of the convergence of the JKO scheme for non-quadratic costs. It was later proven by Cordero, Gangbo, and Houdré in \cite{IneqGeneEntOptTranspo} that for $V$ admitting a convex modulus of convexity $\omega$, one can derive a modified logarithmic Sobolev inequality with $\beta = \omega^*$. In fact, their proof goes beyond the logarithmic case and applies to more general functionals satisfying the McCann condition (we also mention \cite{EntFlowFuncIneqConvSets} by Zugmeyer for a flow approach to such inequalities). We shall explain later more precisely the difference between our approach and Cordero-Gangbo-Houdré's one. In the same line, see also \cite{GeomIneqGenCompPrinIntGas} by Agueh, Ghoussoub, and Kang for other applications. Beyond the log-concave case, still using transport tools, Barthe, Kolesnikov, and Alexander found in \cite{MassTraspVarLogSoboIneq} conditions linked to the integrability behavior of the function $V$. These types of modified logarithmic Sobolev inequalities have also been linked to the Talagrand inequality and concentration phenomena in most of the previously mentioned articles in particular cases, but also in full generality by Gozlan, Roberto, and Samson in a series of articles \cite{HamJacEqMetrSpaceTranspEnt} \cite{CharacTalaTransEntIneqMetSpace} \cite{NewCharacTalaTransEntIneqApp} \cite{ConcLogSoboPoincIneq}.
    
\subsection{The classical Bakry-Emery's Proof}

    Before getting into the proof, and in order to compare our method to the usual continuous flow method, we think it is instructive to recall the formal proof of the logarithmic inequality for $\Lambda$-convex potentials $V$. We omit regularity issues for clarity. The fundamental idea of Bakry and Emery is to consider by how much the Fisher's information $\mathcal{I}[\rho] := \int |\nabla (\log \rho + V)|^2 \dd{\rho}$ and the relative entropy $\mathcal{F}[\rho] = \int (\log \rho + V) \dd{\rho}$ dissipate along the flow generated by $V$, that is, along the solution to the Fokker-Planck's equation
    \begin{equation*}
        \partial_t \rho_t = \Delta \rho_t + \nabla \cdot (\rho_t \nabla V) = \nabla \cdot (\rho_t \nabla [\log \rho_t + V])
    \end{equation*}
    In fact it will be more convenient to work with the density of $\rho_t$ with respect to $\eta = e^{-V}$, i.e. $g_t = \rho_t e^V$, which solves the so-called heat equation associated to the Witten Laplacian:
    \begin{equation*}
        \partial_t g_t = \Delta g_t - \nabla V \cdot \nabla g_t
    \end{equation*}
    This transforms the entropy into $\int g_t \log g_t \dd{\mu}$ and the Fisher's information into $\int |\nabla \log g_t|^2 g_t \dd{\eta}$. One then shows that, for $h_t := \log g_t$
    \begin{equation*}
        \dv{t} \cl{F}[\rho_t] = -\cl{I}[\rho_t] \quad \dv{t} \mathcal{I}[\rho_t] =  -2 \int [||D^2 h_t||_{HS}^2 - D^2 V[\nabla h_t,\nabla h_t]|| g_t \dd{\mu} \leq -2 \Lambda \cl{I}[\rho_t]
    \end{equation*}
    provided that $V$ is $\Lambda$-convex. By Gronwall lemma the second inequality gives $\cl{I}[\rho_t] \leq e^{-2 \Lambda t} \cl{I}[\rho_0]$, and integrating the first one gives
    \begin{equation*}
        \cl{F}[\rho_0] - \cl{F}[\rho_t] = \int_0^t \cl{I}[\rho_s] \dd{s}\leq \frac{1}{2 \Lambda} (1 - e^{-2 \Lambda t}) \cl{I}[\rho_0]
    \end{equation*}
    taking the limit $t \to +\infty$ (using that $\cl{I}[\rho_t] \to 0$ to show the convergence of $\rho_t$ to $\eta$) we deduce the Log-Sobolev inequality
    \begin{equation*}
        \cl{F}[\rho_0] \leq \frac{1}{2 \Lambda} \cl{I}[\rho_0]
    \end{equation*}
    The main ingredients of this method that we will try to mimic in our proofs are the following:
    \begin{enumerate}
        \item A flow driven by a PDE starting from the initial data $\rho_0$. 
        \item An inequality between the dissipation of the entropy and the dissipation of the Fisher's information along this flow. 
        \item The convergence of the flow to the Gibbs measure $e^{-V}$. 
    \end{enumerate}
    In fact, and in contrast with the usual Bakry-Emery theory, we will not need to differentiate twice the entropy, but only to compare the dissipation of two functionals along the flow. Indeed, a careful inspection of Bakry-Emery's proof shows that one only needs an inequality of the form 
    \begin{equation*}
        \Lambda \dv{t} \scr{F}[\rho_t] \geq \dv{t} \scr{I}[\rho_t]
    \end{equation*}
    and convergence results, in order to derive $\scr{F}[\rho_0] \leq \scr{I}[\rho_0]$. That is, $\scr{I}$ does not need to be the dissipation of the entropy. Indeed, this inequality yields $\Lambda (\scr{F}[\rho_0]-\scr{F}[\rho_t]) \leq \scr{I}[\rho_0] - \scr{I}[\rho_t]$, which gives the inequality provided $\scr{F}[\rho_t],\scr{I}[\rho_t]$ converges to $0$ as $t \to +\infty$. 
    
\subsection{Acknowledgment} 

    The authors acknowledge the support of the European Union under the ERC Advanced Grant 101054420 (EYAWKAJKOS). \\
    The authors would also like to thank Filippo Santambrogio for suggesting this problem and the idea of the proof and for useful discussions around technical points, and Ivan Gentil for pointing out various references related to Log-Sobolev inequalities. 

\section{Preliminaries}

    Let us first introduce some basic tools and results that will be needed for our proof. 
    
\subsection{Optimal transportation}

    The content of this section can be found in classical textbooks in this area. We refer for instance to \cite{OTAM}, \cite{VillaniOT} , and \cite{AGS} for the general theory of optimal transportation. 
    
    \begin{definition}[Optimal Transport Problem] \label{def: OT}
        Given $\mu,\nu \in \scr{P}(\Omega)$, the optimal transport cost associated to $h$ between $\mu$ and $\nu$ is defined as
        \begin{equation} \label{eq: definition_OT}
            \scr{T}_h(\mu,\nu) = \min \left \{ \int_{\Omega \times \Omega} h(x-y) \dd{\gamma}, \gamma \in \Pi(\mu,\nu) \right \}
        \end{equation}
        where the minimum is taken over all probability measures $\gamma$ on $\Omega \times \Omega$ with first and second marginals given by $\mu$ and $\nu$ (called transport plans between $\mu$ and $\nu$). 
    \end{definition}
    
    Existence of optimizers follows from the direct method. We shall make extensive use of the dual formulation of the optimal transport problem (see, for instance, \cite{OTAM}, Chapter 1), called Kantorovitch duality, and its important (and highly non-trivial) consequence: the generalized Brenier's theorem. 
    
    \begin{theorem}[Kantorovitch Duality and generalized Brenier's theorem] \label{thm: Kanto_duality}
        The Kantorovitch dual formulation is given by
        \begin{equation} \label{eq: Kanto_duality}
            \scr{T}_h(\mu,\nu) = \max \left \{  \int_\Omega \psi \dd{\mu} + \int_\Omega \phi \dd{\nu}, \psi,\phi \in C(\Omega), \psi(x) + \phi(y) \leq h(x-y) \right \} 
        \end{equation}
        Moreover 
        \begin{enumerate}
            \item The maximum is attained at a (not necessarily unique) pair $(\psi, \phi)$, called Kantorovitch potentials, which are $c$-conjugate to each other, that is
            \begin{equation*} 
                \psi(x) = \inf_{y \in \Omega} h(x-y) - \phi(y) \quad \phi(y) = \inf_{x \in \Omega} h(x-y) - \psi(x) 
            \end{equation*}
            \item If $h$ is $L$-Lipschitz on $\Omega$, then any $c$-conjugate Kantorovitch potentials $\psi,\phi$ are also $L$-Lipschitz. In particular, under our hypotheses on $h$, the Kantorovitch potentials are always Lipschitz with constant only depending on $\Omega$ and $h$.
            
            \item If $\mu \ll \mathcal{L}^d$, then $\psi$ is differentiable $\mu$-a.e. And the unique optimal transport plan is given by $(\id ,T)_\# \mu$, where $T = \id  - \nabla h^*(\nabla \psi)$ (well defined $\mu$-a.e. since $h^*$ is $C^1$ by strict convexity of $h$). $T$ is called the optimal transport map between $\mu$ and $\nu$. 
            \item If we also have $\nu \ll \mathcal{L}^d$, then the optimal transport map $S$ from $\nu$ to $\mu$ satisfies $S \circ T = \id $ $\mu$-a.e. and $T \circ S = \id $ $\nu$-a.e. 
        \end{enumerate}
    \end{theorem}
    
    We shall also use the following inequality, popularized under the name of five gradients inequality, first obtained in \cite{BVEstimates} in the case of the quadratic cost, and generalized by the first author in \cite{FiveGradsIneq} (see also \cite{DiMMurRad} for extension to some class of differentiable manifolds). 
    
    \begin{theorem}[Five gradients inequality] \label{thm: five_gradients_inequality}
        For all $\mu,\nu \in W^{1,1}(\Omega) \cap \scr{P}(\Omega)$, and $(\psi,\phi)$ a choice of Kantorovitch potentials between $\mu$ and $\nu$ there hold
        \begin{equation} \label{eq: five_gradients_inequality}
            \int_\Omega \nabla H(\nabla \psi) \cdot \nabla \mu + \int_\Omega \nabla H(\nabla \phi) \cdot \nabla \nu \geq 0
        \end{equation} 
    \end{theorem}

\subsection{Functionals over probabilities}

    To prove our theorem, we shall make use of the following functionals on probability measures: entropy, energy, and Fisher's information.

    \begin{definition}[Entropy and Potential Energy] \label{def: entropy_energy}
        For $\rho \in \mathcal{P}(\Omega)$ and $V \in W^{1,\infty}(\Omega)$, we define:
        \begin{enumerate}
            \item The entropy of $\rho$ is defined as $\scr{E}(\rho) := \int_\Omega \rho \log \rho , \mathrm{d}x$ if $\rho \ll \cl{L}^d$, and $+\infty$ otherwise.
            \item The potential energy of $\rho$ is defined as $\scr{V}(\rho) := \int_\Omega V , \mathrm{d}\rho$.
            \item The total energy is defined as $\scr{F}(\rho) := \scr{E}(\rho) + \scr{V}(\rho)$.
        \end{enumerate}
    \end{definition}
    
    \begin{definition}[Generalized Fisher's information] \label{def : Fisher_information}
        Let $\rho$ be an absolutely continuous probability measure on $\Omega$ with density. For $\rho > 0$, define the pressure variable $u_\rho := \log \rho + V$. If $u_\rho \in W^{1,\infty}(\Omega)$ and $G : \bb{R}^d \to [0,+\infty)$ is a convex function, we define the $G$-Fisher information by
        \begin{equation} \label{eq: Fisher_information}
            \scr{I}_G(\rho) := \int_\Omega G(\nabla u_\rho) \dd{\rho}
        \end{equation}
    \end{definition}

\subsection{One-Step JKO Scheme}

    The one-step JKO scheme is defined by the the variational problem 
    \begin{equation} \label{eq: one_step_JKO}
        \min_{\rho \in \cl{P}(\Omega)} \scr{F}(\rho) + \scr{T}_h(\rho,\mu) 
    \end{equation}
    for a given reference measure $\mu \in \cl{P}(\Omega)$, where $\scr{F}$ is the total energy defined above (Definition~\ref{def: entropy_energy}). 
    
    \begin{remark}
        TWhile the JKO scheme typically includes a time-step parameter $\tau$, our analysis absorbs it into $h$, as no limiting behavior is required. This parameter will nonetheless reappears in the computations in the quadratic and $p$-power case.  
    \end{remark}

    The following proposition is classical, and most of its ingredients can be found in standard references in the theory.
    
    \begin{theorem} \label{thm: one_step_JKO}
        For any $\mu \in \scr{P}(\Omega)$, the problem
        \begin{equation} \label{eq: one_step_JKO_2}
            \argmin_{\rho \in \scr{P}(\Omega)} \scr{F}(\rho) + \scr{T}_h(\rho,\mu) 
        \end{equation}
        admits a unique solution $\rho$, which we shall also denote by $Q[\mu]$. Furthermore 
        \begin{enumerate}
            \item There exists $\delta > 0$ depending only on $V,h$, and $\Omega$ such that 
            \begin{equation*} \delta \leq \rho \leq \frac{1}{\delta} \end{equation*}
            \item If $(\psi,\phi)$ is a pair of Kantorovitch potentials from $\rho$ to $\mu$, then the optimality condition reads as $\log \rho + V = -\psi + C$ a.e. for some constant $C$ (which can always be chosen to be $0$ up to translation of the potentials). 
            \item In particular, $\log Q[\mu]$ and $Q[\mu]$ are both Lipschitz, with constant depending only on $V,h$, and $\Omega$. 
        \end{enumerate}   
    \end{theorem}
    
    \begin{proof}
        The existence of a solution to the problem can be proved using the direct method in the calculus of variations; indeed, the considered functionals are lower semi-continuous on $\scr{P}(\Omega)$, which is compact (as $\Omega$ is bounded; see, for example, \cite{OTAM}, Chapter 7). $\scr{F}$ is also strictly convex (again, see \cite{OTAM}, Chapter 7), so that the solution is unique. A very slight adaptation of the proof in \cite{BlanchetCarlier} shows that there exists some $\delta > 0$, which depends on $V$, $h$, and $\Omega$, such that $\delta \leq Q[\mu] \leq \frac{1}{\delta}$. This comes from the fact that the function $z \to z \log z$ satisfies the so-called \emph{Inada condition}:
        \begin{align*}
            \lim_{z\xrightarrow{}0^+} f'(z)= -\infty \; \; \; \; \textnormal{and} \; \; \; \; \lim_{z \xrightarrow{}+ \infty} f'(z)= + \infty.
        \end{align*}
        The derivation of optimality conditions can be found again in \cite{OTAM}, Chapter 7 and directly gives the existence of a Kantorovich potential $\psi$ and a constant $C$ such that
        \begin{align*}
            \log Q[\mu] + V = -\psi + C.
        \end{align*}
        This, in turn, implies that $Q[\mu]$ is Lipschitz, as $V$ and $\psi$ are, and that $Q[\mu]$ is bounded away from $0$ and $+\infty$.
    \end{proof}

\subsection{Modulus of convexity and monotonicity}

    We now give the precise definition of the modulus of convexity and monotonicity, which will be used in the proof of our theorem.

    \begin{definition}[Modulus of monotonicity and convexity]
    \label{def: modulus}
        Let $\omega,\sigma : \bb{R}^d \to [0,+\infty)$
        \begin{enumerate}
            \item We say that $\omega$ is a \emph{modulus of monotonicity} of $\nabla V$ if, for all $x,y \in \Omega$, we have
            \begin{equation*}
                (\nabla V(x) - \nabla V(y)) \cdot (x-y) \geq \omega(x-y)
            \end{equation*}
            \item We say that $\sigma$ is a \emph{modulus of convexity} of $V$ if, for any $x,y \in \Omega$, we have 
            \begin{equation*}
                V(x) \geq V(y) + \nabla V(y) \cdot (x-y) + \sigma(x-y)
            \end{equation*}
        \end{enumerate}
    \end{definition}
    
    For simplicity, we shall call $\omega$ modulus of monotonicity of $V$ instead of $\nabla V$. The modulus of monotonicity measures the monotonicity of $\nabla V$ as a monotone vector-valued map, and can be defined for any such map. On the other hand, the modulus of convexity captures how much the graph of $V$ deviates from its supporting hyperplanes. Notice that the modulus of convexity can also be seen as a translation-invariant lower bound on the Bregman divergence of $V$, defined by $D_V(x:y) := V(x) - V(y) - \nabla V(y) \cdot (x-y)$.
    
    It is not hard to see that if $\sigma$ is a modulus of convexity for $V$, then $2 \sigma$ is a modulus of monotonicity; similarly, if $\omega$ is a modulus of monotonicity, then $\int_0^1 t^{-1} \omega(t \cdot) \dd{t}$ is a modulus of convexity. In general, there is no reason to expect those inequalities to be equality. This is the case, for instance, for the function $\frac{1}{p} |\cdot|^p$ with $p \geq 2$ (see Lemma \ref{lemma: modulus_power_p}).

\section{Proof of theorem \ref{thm: log_Sobolev}} 

    First we give a rough overview of the strategy of the proof.
    \begin{enumerate}
        \item Using an improvement of geodesic convexity, we derive a one-step dissipation of entropy along the JKO scheme, involving Fisher-information-like terms.
        \item Similarly, using the five-gradients inequality, we show a one-step dissipation of Fisher information along the JKO scheme.
        \item We then compare these two dissipation inequalities along an iteration of the flow, and identify conditions to obtain a difference inequality between the entropy and the Fisher information.
        \item Finally, we prove that the JKO iteration converges to $\eta$, which will conclude our proof.
    \end{enumerate}

\subsection{One-Step dissipation of entropy}

    The first step is to show a precise result on the dissipation of entropy for the one-step JKO scheme. This is based on the following improved geodesic convexity of the energy functional along Wasserstein geodesics. We shall that this inequality was already present in \cite{IneqGeneEntOptTranspo} in the particular case when $\sigma = \alpha_0 h$ for some $\alpha_0 > 0$.
    
    \begin{proposition}[Improved geodesic convexity] \label{prop: improved_geo_convexity}
        Let $\rho,\mu \in \scr{P}^a(\Omega)$, suppose that $\mu > 0$ and that $u_\mu := \log \mu + V$ belong to $W^{1,\infty}(\Omega)$. Let $T$ be the optimal transport map from $\mu$ to $\rho$. Then if $\sigma$ is a modulus of convexity for $V$ we have 
        \begin{equation} \label{eq: convexity_bound}
            \scr{F}(\rho) \geq \scr{F}(\mu) + \int_\Omega \nabla u_\mu \cdot (\id -T) \dd{\mu} + \int_\Omega \sigma(\id -T) \dd{\mu}
        \end{equation}
    \end{proposition}
    
    \begin{proof} 
        The proof splits into inequalities for the entropy and potential energy.:
        \begin{align*} 
            &\scr{E}(\rho) \geq \scr{E}(\mu) + \int_\Omega \nabla \log \mu \cdot (\id - T) \dd{\mu} \\
            &\scr{V}(\rho) \geq \scr{V}(\mu) + \int_{\Omega} \nabla V \cdot (\id  - T) \dd{\mu} + \int_\Omega \sigma_V(\id  - T) \dd{\mu} 
        \end{align*}
        \begin{enumerate}
            \item The first inequality is classical and follows from the geodesic convexity: formally, if we define the "geodesic interpolation" between $\mu$ and $\rho$ given by $\rho_t := (T_t)_\# \mu$ where $T_t := ((1-t) \id  + t T)$, then $t \to \scr{E}(\rho_t)$ is convex, hence, we have
            \begin{equation*} 
                \scr{E}(\rho) \geq \scr{E}(\mu) + \dv{t}_{|t=0} \scr{E}(\rho_t) 
            \end{equation*}
            Using the fact that $\rho_t$ weakly solves the continuity equation
            \begin{equation*} 
                \partial_t \rho_t + \nabla \cdot (\rho_t v_t) = 0 
            \end{equation*}
            with $v_t = T_t^{-1}(\id - T)$. We then have
            \begin{equation*} 
                \dv{t}_{|t=0} \scr{E}(\rho_t) = \int_\Omega v_0 \cdot \nabla \log \rho_0 \dd{\rho_0} = \int_\Omega \nabla \log \mu \cdot (T-\id ) \dd{\mu} 
            \end{equation*}
            The above computations are formal but can be rigorously justified, this is done for example in Section 2.4 of \cite{Agueh} under some regularity conditions on $\mu,\rho$ and on the cost $h$, which are satisfied here. 
            \item To prove the second inequality, we consider the optimal transport plan $\gamma = (\id ,T)_\# g$ between $g$ and $\rho$, and integrate the inequality $V(x) \geq V(y) + \nabla V(y) \cdot (x-y) + \sigma(x-y)$ against $\gamma$ to conclude. 
        \end{enumerate}
        
    \end{proof}

    \begin{remark}
        This is where our proof starts to differ from Cordero–Gangbo–Houdré's proof \cite{IneqGeneEntOptTranspo}. They directly use the previous inequality for a generic $\mu$ and $\rho = \eta$, combined with Jensen's inequality to conclude.
    \end{remark}
    
    As a consequence we obtain the following dissipation entropy along one step of the JKO scheme.
    
    \begin{proposition}[Dissipation of entropy] \label{prop: dissipation_entropy}
        Suppose $\mu > 0$ is such that $u_\mu \in W^{1,\infty}(\Omega)$. Let $\rho = Q[\mu]$. Let $L$ be any continuous convex function. Then if $V$ admits $\sigma$ as modulus of convexity, we have
        \begin{equation} \label{eq: dissipation_entropy} 
            \scr{F}(\rho) + \int_\Omega L(\nabla u_\mu) \dd{\mu} + \int_\Omega L^*(\nabla h^*(\nabla u_\rho)) \dd{\rho} \geq \scr{F}(\mu) + \int_\Omega \sigma(\nabla h^*(\nabla u_\rho)) \dd{\rho} 
        \end{equation}
    \end{proposition}
    
    \begin{proof} 
        We consider $(\phi,\psi)$ Kantorovitch potentials from $\rho$ to $\mu$, $T$ the associated transport map and $S$ the transport map from $\mu$ to $\rho$. From the previous proposition (\ref{prop: improved_geo_convexity}), we have
        \begin{equation*} 
            \scr{F}(\rho) \geq \scr{F}(\mu) + \int_\Omega \nabla u_\mu \cdot (\id -S) \dd{\mu} + \int_\Omega \sigma(\id -S) \dd{\mu}
        \end{equation*}
        By Jensen’s inequality, the second term is bounded below by
        \begin{equation*} 
            -\int_\Omega L(\nabla u_\mu) \dd{\mu} - \int_\Omega L^*(\id  - S) \dd{\mu} 
        \end{equation*}
        We can then treat the $L^*$ and $\sigma$ terms simultaneously. Using that $g = T_\# \rho$, that $S \circ T = \id $ $\rho$-a.e. and the optimality conditions which give $T - \id  = \nabla h^*(\nabla u_\rho)$ (as $\nabla h^*$ is odd) to deduce
        \begin{align*}
            \int_\Omega [\sigma(\id  - S) - L^*(\id  - S)] \dd{\mu} &= \int_\Omega [\sigma(T - \id) - L^*(T-\id )] \dd{\rho} \\ &= \int_\Omega [\sigma(\nabla h^*(\nabla u_\rho)) - L^*(\nabla h^*(\nabla u_\rho))] \dd{\rho} 
        \end{align*}
        and we conclude.
    \end{proof}

    \begin{remark}
        Notice that by using the flow interchange technic \cite{flowinterchange} (with respect to the flow $\partial_t \rho_t = \nabla \cdot [\rho_t \nabla h^*(\nabla \log \rho_t + \nabla V)]$), we could obtain
        \begin{equation}
            \scr{F}[\mu] \geq \scr{F}[\rho] + \int_\Omega \nabla [\log \rho + V] \cdot \nabla h^*(\nabla [\log \rho + V]) \dd{x}
        \end{equation}
        which is in the reverse order compared to \ref{eq: dissipation_entropy}. As it takes the form $\dv{t} \scr{F}[\rho_t] \leq -\scr{I}[\rho_t]$, it has the wrong sign and hence can't be exploited in our proof. 
    \end{remark}
        
\subsection{One-Step dissipation of Fisher's information}

    From now on we extend the value of $\nabla H$ at $z=0$ by setting $\nabla H(0) = 0$. We note that, as $H,h$ are both radially symmetric of class $C^1(\bb{R}^d \setminus \{ 0 \})$, $\nabla H$ and $\nabla h^*$ are parallel, and $\nabla h^*(z) \neq 0$ for all $z \neq 0$. Therefore we can consider the minimal $\alpha(z) \geq 0$ satisfying $\nabla H(z) = \alpha(z) \nabla h^*(z)$, which is given by $\alpha(z) = \frac{|\nabla H(z)|}{|\nabla h^*(z)|}$ for $z \neq 0$, and $\alpha(0) = 0$. We then have the following dissipation of the Fisher's information. 

    \begin{proposition}[One-Step Dissipation of Fisher's Information] \label{prop: dissipation_information}
        Let $\mu > 0$ with $u_\mu \in W^{1,\infty}(\Omega)$, and $\rho = Q[\mu]$ (so that $\rho$ is absolutely continuous, with $\rho > 0$ and $u_\rho \in W^{1,\infty}(\Omega)$). Then if $V$ admits $\omega$ as a modulus of monotonicity, we have the following dissipation inequality
        \begin{equation} \label{eq: dissipation_information}
            \scr{I}_H(\mu) \geq \scr{I}_H(\rho) + \int_\Omega \alpha(\nabla u_\rho) \omega(\nabla h^*(\nabla u_\rho)) \dd{\rho} 
        \end{equation}
    \end{proposition}
    
    \begin{proof}
        The fact that $\rho > 0$ and $u_\rho \in W^{1,\infty}(\Omega)$ follows from theorem \ref{thm: one_step_JKO}. Consider $(\phi,\psi)$ a pair of Kantorovitch potentials from $\rho$ to $\mu$, with $\phi$ chosen so that $-\phi = u_\rho$. Since $\psi$ is $\mu$-a.e. differentiable one has the inequality
        \begin{equation*} 
            \scr{I}_H(\mu) = \int_\Omega H(\nabla u_\mu) \dd{\mu} \geq \int_\Omega H(\nabla \psi) \dd{\mu} + \int_\Omega \nabla H(\nabla \psi) \cdot (\nabla u_\mu - \nabla \psi) \dd{\mu}
        \end{equation*}
        If $T$ is the optimal transport map from $\rho$ to $\mu$, we can use that $\nabla \psi \circ T = -\nabla \phi$ $\rho$-a.e., which follows from the fact that $T(x)$ attains the maximum of $y \mapsto \psi(y) + \phi(x) - h(x-y)$. Using the symmetry of $H$, we can then write
        \begin{equation*} 
            \int_\Omega H(\nabla \psi) \dd{\mu} = \int_\Omega(\nabla \psi \circ T) \dd{\rho} = \int_\Omega H(-\nabla \phi) \dd{\rho} = \int_\Omega H(\nabla u_\rho) \dd{\rho} = \scr{I}_H(\rho) 
        \end{equation*}
        For the second term, we have
        \begin{equation*}
            -\int_\Omega \nabla H(\nabla \psi) \cdot  \nabla \psi \dd{\mu} = \int_\Omega \nabla H(-\nabla \phi) \cdot \nabla \phi \dd{\rho} = \int_\Omega \nabla H (\nabla \phi) \cdot \nabla u_\rho \dd{\rho}
        \end{equation*}
        so that 
        \begin{align*}
            \int_\Omega \nabla H (\nabla \psi) \cdot ( \nabla u_\mu - \nabla \psi) \dd{\mu} &= \int_\Omega \nabla H(\nabla \psi) \cdot \nabla u_\mu \dd{\mu} + \int_\Omega \nabla H(\nabla \phi) \cdot \nabla u_\rho \dd{\rho} \\
            &= \int_\Omega \nabla H(\nabla \psi) \cdot \nabla \mu \dd{x} + \int_\Omega \nabla H(\nabla \phi) \cdot \nabla \rho \dd{x} \\
            &+ \int_\Omega \nabla H(\nabla \psi) \cdot \nabla V \dd{\mu} + \int_\Omega \nabla H(\nabla \phi) \cdot \nabla V \dd{\rho}
        \end{align*}
        By the five-gradients inequality \ref{eq: five_gradients_inequality} with the convex function $H$ and the Kantorovich pair $(\phi,\psi)$ we have
        \begin{align*}
            \int_\Omega \nabla H (\nabla \phi) \cdot  \nabla{\rho} \dd{x} + \int_\Omega \nabla H (\nabla \psi) \cdot \nabla \mu \dd{x} \geq 0.
        \end{align*}
        Furthermore
        \begin{equation*}
            \int_\Omega \nabla H(\nabla \psi) \cdot \nabla V \dd{\mu} = \int_\Omega \nabla H(-\nabla \phi) \cdot \nabla V \circ T \dd{\rho} = -\int_\Omega \nabla H(\nabla \phi) \cdot \nabla V \circ T \dd{\rho}
        \end{equation*}
        Hence the second term reduces to
        \begin{align*}
            \int_\Omega \nabla H(-\nabla \phi) \cdot [\nabla V \circ T- \nabla V] \dd{\rho} &= \int_\Omega \alpha(-\nabla \phi) \nabla h^*(-\nabla \phi) \cdot [\nabla V \circ T- \nabla V] \dd{\rho} \\
            &= \int_\Omega \alpha(-\nabla \phi) (T-\id) \cdot [\nabla V \circ T- \nabla V] \dd{\rho} \\
            &\geq \int_\Omega \alpha(-\nabla \phi) \omega(T-\id) \dd{\rho} = \int_\Omega \alpha(\nabla u_\rho) \omega(\nabla h^*(\nabla u_\rho)) \dd{\rho} \\
            &= \int_\Omega \alpha(\nabla u_\rho) \omega(\nabla h^*(\nabla u_\rho) \dd{\rho}
        \end{align*}
        which gives the final inequality.
    \end{proof}

    \begin{remark}
        Such an estimate has already been studied, in the case of the quadratic cost, by Di Marino and Santambrogio in \cite{FokkerPlanckLp}, to derive Sobolev estimates in the JKO scheme, and in \cite{Caillet_Continuity} by Santambrogio and the first order to show propagation of modulus of continuity along the JKO. Our proof strategy follows their approach, with obvious modifications to account for the modified cost.
    \end{remark}

\subsection{Iteration along the flow}

    We now consider an initial data $\rho_0 \in \scr{P}(\Omega)$ absolutely continuous, with $\rho_0 > 0$, and $u_{\rho_0} \in W^{1,\infty}(\Omega)$. We define a discrete flow $(\rho_k)_{k \geq 0}$ by iteration of the one-step JKO scheme. That is
    \begin{equation*} 
        \rho_{k+1} \in \argmin \scr{F} + \scr{T}_h(\cdot,\rho_k) 
    \end{equation*}
    By the regularity of the one-step JKO scheme, we have $\rho_k > 0$ for all $k \geq 1$ and $u_{\rho_k} \in W^{1,\infty}(\Omega)$. Hence, we can apply the two dissipation inequalities. We denote $u_k := u_{\rho_k}$. From the one-step dissipation of entropy and Fisher's we derive the following two estimates
    
    \begin{proposition}[Dissipation of Entropy and Fisher's along the flow] \label{prop: dissipation_flow}
        Let's define
        \begin{align*}
            &R_k^{\rm{ent}} := \int_\Omega \left [ L(\nabla u_k) + L^*(\nabla h^*(\nabla u_k)) - \sigma(\nabla h^*(\nabla u_k)) \right] \dd{\rho_k} \\
            &R_k^{\rm{inf}} := \int_\Omega \alpha(\nabla u_k) \omega(\nabla h^*(\nabla u_k)) \dd{\rho_k}
        \end{align*}
        Then, for all $n \geq 0$ we have
        \begin{align*}  
            &\scr{I}_L(\rho) - \scr{I}_L(\rho_n) + \sum_{k=1}^n R_k^{\rm{ent}} \geq \scr{F}(\rho) - \scr{F}(\rho_n) \\
            &\scr{I}_H(\rho_0) - \scr{I}_H(\rho_n) - \sum_{k=1}^n R_k^{\rm{inf}} \geq 0
        \end{align*}
        and as a consequence
        \begin{equation} \label{eq: almost_log_sobo}
            \scr{F}(\rho) - \scr{F}(\rho_n) \leq \scr{I}_G(\rho) - \scr{I}_G(\rho_n) + \sum_{k=1}^n \Delta_k
        \end{equation}
        with $\Delta_k = R_k^{\rm{ent}} - R_k^{\rm{inf}}$ and $G = H + L$. 
    \end{proposition}

    \begin{proof}
        We prove each inequality separately:
        \begin{enumerate} 
            \item By the one-step dissipation of entropy \ref{eq: dissipation_entropy} we have
            \begin{equation*}
                \scr{F}(\rho_{k+1}) + \int_\Omega L(\nabla u_k) \dd{\rho_k} + \int_\Omega L^*(\nabla h^*(\nabla u_{k+1})) \dd{\rho_{k+1}} \geq \scr{F}(\rho_k) + \int_\Omega \sigma(\nabla h^*(\nabla u_{k+1}) \dd{\rho_{k+1}}
            \end{equation*}
            which gives
            \begin{equation*}
                \scr{F}(\rho_k) - \scr{F}(\rho_{k+1}) \leq \scr{I}_L(\rho_k) - \scr{I}_L(\rho_{k+1}) + R_{k+1}^{\rm{ent}}
            \end{equation*}
            Summing these inequalities from $k=0$ to $n-1$ gives the first inequality. 
            \item By the one-step dissipation of Fisher's information \ref{eq: dissipation_information} we have
            \begin{equation*}
                \scr{I}_H(\rho_k) \geq \scr{I}_H(\rho_{k+1}) + \int_\Omega \alpha(\nabla u_{k+1}) \omega(\nabla h^*(\nabla u_{k+1})) \dd{\rho_{k+1}} 
            \end{equation*}
            Again, summing the inequalities from $k=0$ to $n-1$ gives the second inequality. 
            \item Adding up the two previous inequalities gives the final one. 
        \end{enumerate}
    \end{proof}
    
\subsection{Convergence of the flow to the equilibrium}
    
    To conclude, we would like to pass to the limit as $n \to +\infty$ in the previous identity. To do this, we shall first prove that $\rho_n \to \eta$ in a sufficiently strong sense as $n \to +\infty$. For that, we need the following stability result for the JKO scheme.

    \begin{lemma}[Stability of the JKO-scheme] \label{lemma: stability_JKO}
        Let $\mu_n \to \mu$ weakly on $\cl{P}(\Omega)$. Then $Q[\mu_n] \to Q[\mu]$ uniformly on $\overline{\Omega}$. 
    \end{lemma}
    
    \begin{proof}
        By theorem \ref{thm: one_step_JKO} there exists a universal $\delta$ such that $\delta \leq Q[\mu_n] \leq \delta^{-1}$ and $Q[\mu_n]$ is uniformly Lipschitz. By Arzelà-Ascoli theorem, there exists a uniformly converging subsequence. Then for any $\rho \in \scr{P}(\Omega)$, passing to the limit along a converging subsequence in the inequality 
        \begin{equation*}
            \scr{F}(Q[\mu_n]) + \scr{T}_h(Q[\mu_n],\mu_n) \leq \scr{F}(\rho) + \scr{T}_h(\rho,\mu_n) 
        \end{equation*}
        using joint continuity of $\cl{T}_h$, since $\Omega$ is pre-compact and $h$ continuous (see theorem 1.51 \cite{OTAM}), and the fact that $\scr{F}$ is l.s.c. (see for instance \cite{OTAM} proposition 7.1 and 7.7), we deduce that any limit point is a minimizer of $\scr{F} + \scr{T}_h(\cdot,\mu)$, hence is equal to $Q[\mu]$ by uniqueness. Therefore the whole sequence is converging uniformly to $Q[\mu]$. 
    \end{proof}
    
    \begin{proposition}[Convergence to equilibrium] \label{prop: convergence_equlibrium}
        We have $\rho_n \to \eta$ and $u_n \to 0$ uniformly on $\Omega$ when $n \to +\infty$.  
    \end{proposition}
    
    \begin{proof}
        By Theorem \ref{thm: one_step_JKO} we can bound $(\rho_n)_{n \geq 1}$ and $(u_n)_{n \geq 1}$ uniformly in Lipschitz norm. Hence we can find a subsequence $k_n \to +\infty$ such that $\rho_{k_n} \to \rho_\infty$ and $u_{k_n} \to u_\infty$ uniformly. 
        
        By optimality of $\rho_{n+1}$, we also have
        \begin{equation*} 
            \scr{F}(\rho_{n+1}) + \scr{T}_h(\rho_{n+1},\rho_n) \leq \scr{F}(\rho_n) 
        \end{equation*}
        By optimality of $\rho_{n+1}$. Hence $\scr{F}(\rho_n) - \scr{F}(\rho_{n+1}) \geq \scr{T}_h(\rho_{n+1},\rho_n) \geq 0$, but since :
        \begin{equation*} 
            \sum_{k=0}^N \scr{F}(\rho_n) - \scr{F}(\rho_{n+1}) = \scr{F}(\rho_0) - \scr{F}(\rho_{N+1}) \leq \scr{F}(\rho_0) - \inf \scr{F} 
        \end{equation*}
        we deduce that $\scr{F}(\rho_n) - \scr{F}(\rho_{n+1})$ is summable, hence goes to $0$ as $n \to +\infty$. But since $\rho_{n+1} = Q[\rho_n]$, by continuity we have $\rho_{k_n+1} = Q[\rho_{k_n}] \to Q[\rho_\infty]$. Using that $\scr{F}$ is continuous for the uniform convergence, we obtain that $\scr{F}(\rho_\infty) = \scr{F}(Q[\rho_\infty])$, which implies that $\mathscr{T}_h(\rho_\infty,Q[\rho_\infty]) = 0$ by the inequality $\scr{F}(Q[\rho_\infty]) + \mathscr{T}_h(\rho_\infty,Q[\rho_\infty])  \leq \scr{F}(\rho_\infty)$. As $\mathscr{T}_h(\rho_\infty,Q[\rho_\infty]) = 0$ implies that $T x = x$ $\rho_\infty$-a.e. ($T$ being the optimal transport map from $Q[\rho_\infty]$ to $\rho_\infty$), since $h(z) = 0$ iff $z=0$, this implies that $Q[\rho_\infty] = \rho_\infty$, i.e. $\rho_\infty$ is a fixed point of $Q$. \\
        In particular, Kantorovitch potentials from $Q[\rho_\infty]$ to $\rho_\infty$ are constant. By the optimality condition for the one-step JKO scheme, this implies $\log Q[\rho_\infty] + V = -C$ for some $C$, hence $\rho_\infty = Q[\rho_\infty] = e^{-C - V} = e^{-C} \eta$, and integrating implies that $C=0$, that is the only fixed point is $\eta$. Since this is true for any subsequence, it is in fact true for the whole sequence. \\ 
        Furthermore, passing to the limit to the limit in $u_{k_n} = \log \rho_{k_n} + V$ implies that $u_\infty = 0$. 
    \end{proof}

\subsection{End of the proof}

    We can now complete the proof of Theorem \ref{thm: log_Sobolev}, assuming first that $\Omega$ is bounded. Under the hypotheses of our main theorem, since
    \begin{equation*}
        \Delta_k = \int_\Omega [L(\nabla u_k) + L^*(\nabla h^*(\nabla u_k)) - \sigma(\nabla h^*(\nabla u_k)) - \alpha(\nabla u_k) \omega(\nabla h^*(\nabla u_k))] \dd{\rho_k}
    \end{equation*}
    we then have $\Delta_k \leq 0$ provided that
    \begin{equation*}
        L(z) + L^*(\nabla h^*(z)) \leq \sigma(\nabla h^*(z)) + \alpha(z) \omega(\nabla h^*(z))
    \end{equation*}
    This is exactly the hypothesis of theorem \ref{thm: log_Sobolev}. It follows that 
    \begin{equation*}
        \scr{F}(\rho_0) - \scr{F}(\rho_n) \leq \scr{I}_G(\rho_0) - \scr{I}_G(\rho_n)
    \end{equation*}
    for all $n \geq 0$. Since $\rho_n \to \eta$ and $u_n \to 0$ uniformly, we have $\scr{F}(\rho_n) \to \scr{F}(\eta) = 0$. Let $\gamma_n$ be the optimal coupling between $\rho_n$ and $\rho_{n-1}$, given by the map $T_n$. Then, $\nabla u_n(x) = \nabla h(T_n x - x)$ $\rho_n$-a.e. Hence
    \begin{align*}
        \scr{I}_G(\rho_n) = \int_\Omega G(\nabla u_n) \dd{\rho_n} = \int_{\Omega \times \Omega} G(\nabla h(x-y)) \dd{\gamma_n}
    \end{align*}
    As that the set of all couplings with marginals given by $(\rho_n,\rho_{n-1})$ for some $n$ is weakly compact, and continuity of the transport cost, we can show that up to extraction, $\gamma_n$ converges to some optimal coupling between $\eta$ and itself, hence to $(\id,\id)_\# \eta$. Since $G \circ \nabla h$ is continuous, we have $\scr{I}_G(\rho_n) \to G(0) = 0$. This shows that
    \begin{equation*}
        \scr{F}(\rho_0) \leq \scr{I}_G(\rho_0)
    \end{equation*}
    The form of the Log-Sobolev inequality explicitly stated in the theorem can be easily obtained by replacing $\rho_0$ by its log-density with respect to the measure $\eta$, i.e. by $g = u_{\rho_0} = \log \rho_0 + V$. This yields the inequality for all $g$ with $\int_\Omega e^g \dd{\eta} = 1$ with $\nabla g \in W^{1,\infty}(\Omega)$, and concludes our proof.

    For the general case of $\Omega$ unbounded, we can approximate $\Omega$ by $\Omega_n$ sequence of bounded convex domains, using that the condition on $V$ doesn't depends on $n$, and by approximation we obtain the inequality for $\Omega$.

\section{Applications}

\subsection{Simpler conditions}

    By taking $L = h^*$ and $H = C h^*$ in the hypotheses of Theorem \ref{thm: log_Sobolev}, we obtain the following simpler set of conditions.

    \begin{proposition}[A simpler condition] \label{prop: simpler_condition}
        Suppose that $V$ has modulus $\omega,\sigma$, and that for some cost function $h$, and constant $C \geq 0$ we have
        \begin{equation} \label{eq: simpler_condition}
            \nabla h(z) \cdot z \leq \sigma(z) + C \omega(z)
        \end{equation}
        Then $\eta = e^{-V}$ satisfies the modified Log-Sobolev inequality 
        \begin{equation} \label{eq: simpler_condition_log_sobo}
            \int g \log g \dd{\eta} \leq (1 + C) \int h^*(\nabla \log g) g \dd{\eta}
        \end{equation}
    \end{proposition}

    This is particularly interesting if $\sigma$ and $\omega$ are not convex, in which case the result of Cordero-Gangbo-Houdré \cite{IneqGeneEntOptTranspo} does not apply. On the other hand, when $\sigma$ and $\omega$ are proportional to the cost function $h$, our inequality is worse than theirs. Indeed, it requires bounds of the form $\nabla h(z) \cdot z \leq A h(z)$ for some $A>0$, which is equivalent to $h(\lambda z) \leq \lambda^A h(z)$ for all $\lambda \geq 1$, i.e., a sub-homogeneity condition on $h$. Therefore, our result does not fully recover Cordero-Gangbo-Houdré's inequality.

    If we restrict to radial moduli and cost, then inequality takes a simpler form, and one can derive the following corollary:

    \begin{corollary}[Radialy symmetric case] \label{coro: radially_sym_case}
        Suppose that $\sigma(z) = \sigma_r(|z|)$ and $\omega(z) = \omega_r(|z|)$, and assume that there exists a strictly increasing function $\theta : \bb{R}_+ \to \bb{R}$ and a constant $C \geq 0$ such that 
        \begin{equation} \label{eq: radial_sym_cond}
            \theta \leq \sigma_r + C \omega_r
        \end{equation}
        Then setting $l(t) := \int_0^t \theta^{-1}(s) \dd{s}$ the following Log-Sobolev inequality holds
        \begin{equation} \label{eq: radial_sym_log_sobo}
            \int_\Omega g e^g \dd{\eta} \leq (1+C) \int_\Omega l(|\nabla g|) e^g \dd{\eta}
        \end{equation}
    \end{corollary}
    
    \begin{proof}
        Set $h(z) := \int_0^{|z|} \theta(s) \dd{s} = \kappa(|z|)$, then using that $\theta$ is strictly increasing, simple computations show that $h^*(z) = \kappa^*(|z|)$ where $\kappa^*(t) = \sup_{s \geq 0} s t - \kappa(s)$. Then using invertibility of $\theta$, one easily computes $\kappa^*(t) = \int_0^t \theta^{-1}(s) \dd{s}$. It is easy to see that $h$ satisfies the hypothesis of being a cost function. Hence applying the previous proposition (\ref{prop: simpler_condition}) concludes. 
    \end{proof}

\subsection{Classical logarithmic Sobolev inequality}

    Our theorem recovers the classical logarithmic Sobolev inequality for strongly log-concave measure. As mentioned in the introduction, we observe that we do not to restrict ourself to the choice of the quadratic cost, corresponding to the JKO scheme associated to the Fokker-Planck equation, to obtain the inequality. Indeed, the proof will show that any cost of the form $h_\tau(x)=\tau h\left({x}{\tau}\right)$ with $h$ convex increasing works, and letting $\tau$ goes to $0$ will give the result. This suggests that the same proof might be possible at the continuous level, using the flow $\partial_t \rho_t = \nabla \cdot [\rho_t \nabla h^*(\nabla \log \rho_t + \nabla V)]$ for any nice enough convex function $h$.  
    
    \begin{proposition}
        Under the our hypotheses on $V$, if $V$ is additionally $\Lambda$-convex, then we have the classical Log-Sobolev inequality:
        \begin{equation}
            \int_\Omega g e^g \dd{\eta} \leq \frac{1}{2 \Lambda} \int |\nabla g|^2 e^g \dd{\eta}
        \end{equation}
    \end{proposition}
    
    \begin{proof}
        Consider any $h$ satisfying the hypothesis of the introduction. Let $H = A |x|^2$ with $A$ to be chosen later on, so that $\alpha = 2 A |x| / |\nabla h^*|$, take $L = h^*$ so that the inequality reads as $\nabla h^* \cdot x \leq A \Lambda |x| \, |\nabla h^*|$ which is true as long as $2 A \Lambda \geq 1$. The minimal $A$ is then given by $1 / 2 \Lambda$. Hence we obtain
        \begin{equation*}
            \mathrm{Ent}_\mu(e^g) \leq \frac{1}{2 \Lambda} \int |\nabla g|^2 e^g \dd{\mu} + \tau \int h^*(|\nabla g|) e^g \dd{\mu}
        \end{equation*}
        If we let $\tau \to 0$ we obtain the inequality. 
    \end{proof}
    
    \begin{remark}
        As explained above, this amounts to compare the entropy and the Fisher's information along the flow of the equation $\partial_t \rho_t = \nabla \cdot [\rho \nabla h^*(\nabla \log \rho + \nabla V)]$. Indeed, this choice of $h^*_\tau$ corresponds to the JKO scheme with step $\tau$ associated to $h^*$. That is, if $\rho_{k+1}^\tau \in \argmin \rm{Ent}_\mu(\rho) + \scr{T}_{h_\tau}(\rho,\rho_k^\tau)$, then when $\tau \to 0$, the affine interpolation $\rho_t^\tau = \rho_k^\tau$ on $[k \tau,(k+1)\tau)$ converges to the solution of the flow. We expect that the inequality can be recovered by proving that the derivative of the entropy along the flow is comparable to the derivative of the Fisher's information, we do not pursue this here. 

        We can also now explain why we consider our method as interpolating between Bakry-Emery's method and Cordero-Gangbo-Houdré's method. If we consider the JKO scheme with cost $h_\tau = \tau h(\tau^{-1} \cdot)$ for a given cost $h$. Then as $\tau \to 0$, the JKO converges to the continuous flow equation, and our discrete flow method is then completely analogue to Bakry-Emery as the above example suggest. In the other hand, if we take $\tau \to +\infty$, we formally only have to perform one-step of the JKO to go from a measure to the global minimizer of the energy, and we only exploit the geodesic-convexity of the entropy to obtain the result. \\
        This formal $\tau \to +\infty$ case is also going to be the case when working with modulus of $p$-power type, we refer to the appendix for this case.
    \end{remark} 

\bibliographystyle{plain}
\bibliography{Ref}

\begin{thebibliography}{10}

\bibitem{Agueh}
Martial Agueh.
\newblock Existence of solutions to degenerate parabolic equations via the
  {Monge}-{Kantorovich} theory.
\newblock {\em Adv. Differ. Equ.}, 10(3):309--360, 2005.

\bibitem{GeomIneqGenCompPrinIntGas}
Martial Agueh, Nassif Ghoussoub, and Xiaosong Kang.
\newblock Geometric inequalities via a general comparison principle for
  interacting gases.
\newblock {\em Geometric and Functional Analysis}, 14:215--244, 02 2004.

\bibitem{AGS}
Luigi Ambrosio, Nicola Gigli, and Giuseppe Savar{\'e}.
\newblock {\em Gradient flows in metric spaces and in the spaces of probability
  measures}.
\newblock Lectures in Mathematics. Birkh\"auser, 2008.

\bibitem{Bakry1}
Dominique Bakry.
\newblock Hypercontractivity and its usage in semigroup theory.
\newblock In {\em Lectures on probability theory. Ecole d'Et\'e de
  Probabilit\'es de Saint- Flour XXII-1992. Summer School, 9th-25th July, 1992,
  Saint-Flour, France}, pages 1--114. Berlin: Springer, 1994.

\bibitem{LivreIvan}
Dominique Bakry, Ivan Gentil, and Michel Ledoux.
\newblock {\em Analysis and geometry of {Markov} diffusion operators}, volume
  348 of {\em Grundlehren Math. Wiss.}
\newblock Cham: Springer, 2014.

\bibitem{MassTraspVarLogSoboIneq}
Franck Barthe and Alexander~V. Kolesnikov.
\newblock Mass transport and variants of the logarithmic sobolev inequality.
\newblock {\em Journal of Geometric Analysis}, 18(4):921--979, 2008.

\bibitem{ModLogSoboIneqR}
Franck Barthe and Cyril Roberto.
\newblock {Modified logarithmic Sobolev inequalities on $mathbb{R}$}.
\newblock {\em {Potential Analysis}}, 29(2):167--193, June 2008.

\bibitem{BlanchetCarlier}
Adrien Blanchet and Guillaume Carlier.
\newblock Optimal transport and cournot-nash equilibria.
\newblock {\em Mathematics of Operations research}, pages 125--145, 2015.

\bibitem{PoinIneqTalagConcPhenoExpoDitrib}
S.~Bobkov and M.~Ledoux.
\newblock Poincar{\'e}'s inequalities and talagrand's concentration phenomenon
  for the exponential distribution.
\newblock {\em Probability Theory and Related Fields}, 107(3):383--400, 1997.

\bibitem{BobkovLedoux}
Sergey Bobkov and M.~Ledoux.
\newblock {From Brunn-Minkowski To Brascamp-Lieb And To Logarithmic Sobolev
  Inequalities}.
\newblock {\em Geometric and Functional Analysis}, 10, 12 2000.

\bibitem{EntBoundIsop}
S.G. Bobkov and B.~Zegarlinski.
\newblock {\em Entropy Bounds and Isoperimetry}.
\newblock Memoirs of the American Mathematical Society. American Mathematical
  Society, 2005.

\bibitem{FiveGradsIneq}
Thibault Caillet.
\newblock The five gradients inequality for non quadratic costs.
\newblock {\em Comptes Rendus. Math\'ematique}, 361:715--721, 2023.

\bibitem{Caillet_Continuity}
Thibault Caillet and Filippo Santambrogio.
\newblock Fisher information and continuity estimates for nonlinear but
  1-homogeneous diffusive {PDE}s (via the {JKO} scheme).
\newblock {\em Bulletin of the Hellenic Mathematical Society}, 2025.
\newblock cvgmt preprint.

\bibitem{AppMassTransGaussTypeIneq}
Dario Cordero-Erausquin.
\newblock Some applications of mass transport to gaussian-type inequalities.
\newblock {\em Archive for Rational Mechanics and Analysis}, 161(3):257--269,
  2002.

\bibitem{IneqGeneEntOptTranspo}
Dario Cordero-Erausquin, Wilfrid Gangbo, and Christian Houdré.
\newblock Inequalities for generalized entropy and optimal transportation.
\newblock {\em Contemp. Math.}, 353, 05 2003.

\bibitem{flowinterchange}
Robert J.~McCann Daniel~Matthes and Giuseppe Savar\'e'.
\newblock A family of nonlinear fourth order equations of gradient flow type.
\newblock {\em Communications in Partial Differential Equations},
  34(11):1352--1397, 2009.

\bibitem{BVEstimates}
Guido De~Philippis, Alp{\'a}r~Rich{\'a}rd M{\'e}sz{\'a}ros, Filippo
  Santambrogio, and Bozhidar Velichkov.
\newblock ${BV}$ estimates in {Optimal} {Transportation} and {Applications}.
\newblock {\em Arch. Ration. Mech. Anal.}, 212(2):829--860, 2020.

\bibitem{DiMMurRad}
Simone {Di Marino}, Simone Murro, and Emanuela Radici.
\newblock The five gradients inequality on differentiable manifolds.
\newblock {\em Journal de Mathématiques Pures et Appliquées}, 187:294--328,
  2024.

\bibitem{FokkerPlanckLp}
Simone Di~Marino and Filippo Santambrogio.
\newblock {JKO} estimates in linear and non-linear {Fokker}-{Planck} equations,
  and {Keller}-{Segel}: ${L}^p$ and {Sobolev} bounds.
\newblock {\em Ann. Inst. H. Poincaré Anal. Non Linéaire},
  39(39):1485–1517, 2022.

\bibitem{GeneOptLpEucLogSoboIneqHamJacEq}
Ivan Gentil.
\newblock The general optimal lp-euclidean logarithmic sobolev inequality by
  hamilton--jacobi equations.
\newblock {\em Journal of Functional Analysis}, 202(2):591--599, 2003.

\bibitem{GentilGuillinMiclo1}
Ivan Gentil, Arnaud Guillin, and Laurent Miclo.
\newblock Modified logarithmic {Sobolev} inequalities and transportation
  inequalities.
\newblock {\em Probab. Theory Relat. Fields}, 133(3):409--436, 2005.

\bibitem{ModLogSoboIneqNullCurv}
Ivan Gentil, Arnaud Guillin, and Laurent Miclo.
\newblock {Modified logarithmic Sobolev inequalities in null curvature}.
\newblock {\em {Revista Matem{\'a}tica Iberoamericana}}, 23(1):237--260, 2007.

\bibitem{NewCharacTalaTransEntIneqApp}
Nathael Gozlan, Cyril Roberto, and Paul-Marie Samson.
\newblock {A new characterization of Talagrand’s transport-entropy
  inequalities and applications}.
\newblock {\em The Annals of Probability}, 39(3), May 2011.

\bibitem{ConcLogSoboPoincIneq}
Nathael Gozlan, Cyril Roberto, and Paul-Marie Samson.
\newblock {From concentration to logarithmic Sobolev and Poincaré
  inequalities}.
\newblock {\em {Journal of Functional Analysis}}, 260(5):1491--1522, 2011.

\bibitem{CharacTalaTransEntIneqMetSpace}
Natha{\"e}l Gozlan, Cyril Roberto, and Paul-Marie Samson.
\newblock {Characterization of Talagrand's transport-entropy inequalities in
  metric spaces.}
\newblock {\em {Annals of Probability}}, Vol. 41(No. 5):3112--3139, 2013.

\bibitem{HamJacEqMetrSpaceTranspEnt}
Nathael Gozlan, Cyril Roberto, and Paul-Marie Samson.
\newblock {Hamilton Jacobi equations on metric spaces and transport-entropy
  inequalities}.
\newblock {\em {Revista Matem{\'a}tica Iberoamericana}}, 30(1):133--163, 2014.

\bibitem{Gross1}
Leonard Gross.
\newblock Logarithmic {Sobolev} inequalities.
\newblock {\em Am. J. Math.}, 97:1061--1083, 1976.

\bibitem{LecNoteLogSoboIneq}
A.~Guionnet and B.~Zegarlinski.
\newblock Lectures on {Logarithmic} {Sobolev} {Inequalities}.
\newblock {\em S\'eminaire de probabilit\'es de Strasbourg}, 36:1--134, 2002.

\bibitem{JKO}
Richard Jordan, David Kinderlehrer, and Felix Otto.
\newblock {The Variational Formulation of the Fokker--Planck Equation}.
\newblock {\em SIAM Journal on Mathematical Analysis}, 29(1):1--17, 1998.

\bibitem{Ledoux1}
Michel Ledoux.
\newblock Concentration of measure and logarithmic {Sobolev} inequalities.
\newblock In {\em S\'eminaire de probabilit\'es XXXIII}, pages 120--216.
  Berlin: Sprin\-ger, 1999.

\bibitem{McCannCondition}
Robert~J. McCann.
\newblock A convexity principle for interacting gases.
\newblock {\em Advances in Mathematics}, 128(1):153--179, 1997.

\bibitem{GeneIneqTalaLinkLogSoboIneq}
F.~Otto and C.~Villani.
\newblock {Generalization of an Inequality by Talagrand and Links with the
  Logarithmic Sobolev Inequality}.
\newblock {\em Journal of Functional Analysis}, 173(2):361--400, 2000.

\bibitem{InitLogSoboIneq}
G.~Royer and Donald~G. Babbitt.
\newblock {An Initiation to Logarithmic Sobolev Inequalities}, 2007.

\bibitem{OTAM}
Filippo Santambrogio.
\newblock {\em Optimal {Transport} for {Applied} {Mathematicians}}, volume~87
  of {\em Progress in Nonlinear Differential Equations and their Applications}.
\newblock Birkh\"auser, 2015.

\bibitem{VillaniOT}
C{\'e}dric Villani.
\newblock {\em Topics in {Optimal} {Transportation}}, volume~58 of {\em
  Graduate Studies in Mathematics}.
\newblock American Mathematical Society, 2003.

\bibitem{VillaniO&N}
C{\'e}dric Villani.
\newblock {\em Optimal {Transport}, {Old} and {New}}, volume 338 of {\em
  Grundlehren der Mathematischen Wissenschaften}.
\newblock Springer, 2009.

\bibitem{EntFlowFuncIneqConvSets}
Simon Zugmeyer.
\newblock {Entropy flows and functional inequalities in convex sets}.
\newblock working paper or preprint, December 2019.

\end{thebibliography}

\appendix

\section{Modulus of p-power type}

    In this section we consider potential $V$ admitting a $p$-power like moduli with $p \geq 2$, that is both the modulus of convexity and monotonicity of $V$ are proportional to $|\cdot|^p$. This is typically the case for $V = \frac{1}{p} |\cdot|^p$ where the modulus is given by $2^{2-p} |\cdot|^p$ (see lemma \ref{lemma: modulus_power_p} below). 

    We will see that our approach recover an inequality already obtained by Cordero-Gangbo-Houdré \cite{IneqGeneEntOptTranspo} and Bobkov-Ledoux \cite{BobkovLedoux}. Even though we cannot obtain a new inequality, we still believe that the computation illustrate our method well, and why it interpolates between Bakry-Emery and Cordero-Gangbo-Houdré, that's why we decided to keep the computations in the appendix.

    To tackle this case, we redo the proof deriving explicit dissipation inequalities adapted to the $p$-power case. We shall use the cost $h_p := \frac{1}{p} |\cdot|^p$ and denote by $\scr{I}_q$ the $q h_q$-Fisher's information, i.e. $\scr{I}_q(\rho) = \int_\Omega |\nabla u_\rho|^q \dd{\rho}$. We then have the following

    \begin{proposition}[$p$-power improvement of dissipation] \label{prop: p_power_dissipation}
        Then if $\rho = Q[\mu]$ using the cost $h^*_{\tau,q} = \tau h_q$, for $\tau > 0$, and $q$ conjugate exponent of $p$ we have:
        \begin{enumerate}
            \item If $V$ has modulus of convexity $\sigma =  \frac{\alpha}{p} |\cdot|^p$ we have
            \begin{equation} \label{eq: p_dissipation_entropy}
                \scr{F}(\rho) + \tau \scr{I}_q(\mu)^\frac{1}{q} \scr{I}_q(\rho)^\frac{1}{p} \geq \scr{F}(\mu) + \tau^p \frac{\alpha}{p} \scr{I}_q(\rho)
            \end{equation}
            which implies
            \begin{equation} \label{eq: p_dissipation_entropy_v2}
                \scr{F}(\rho) + \frac{1}{q} \frac{1}{\alpha^{q-1}} \scr{I}_q(\mu) \geq \scr{F}(\mu)
            \end{equation}
            \item If $V$ has modulus of monotonicity $\omega = \frac{\beta}{q} |\cdot|^p$ we have
            \begin{equation} \label{eq: p_dissipation_Fisher}
                \scr{I}_q(\mu) \geq (1+\tau^{p-1} \beta) \scr{I}_q(\rho)
            \end{equation}
        \end{enumerate}
    \end{proposition}

    \begin{proof}
        Let's tackle the inequalities separately.
        \begin{enumerate}
            \item For the first one we use $L = \frac{A^q}{q} |z|^q$ so that $L^* = \frac{1}{p A^p} |z|^p$. Using also that $|\nabla h^*_{\tau,q}| = \tau |z|^{q-1}$, we have
            \begin{align*}
                \int_\Omega L^*(\nabla h^*_{\tau,q}(\nabla u_\rho)) \dd{\rho} &= \frac{1}{A^p p} \tau^p \int_\Omega |\nabla u_\rho|^{p(q-1)} \dd{\rho} = \frac{1}{A^p p} \tau^p \int_\Omega |\nabla u_\rho|^q \dd{\rho} = \frac{1}{A^p p} \tau^p \scr{I}_q(\rho)
            \end{align*}
            Plugging this into the dissipation of entropy \ref{prop: dissipation_entropy}, and using $\sigma(\nabla h^*_{\tau,q}(z)) = \tau^p \frac{\alpha}{p} |z|^q$, we end up with
            \begin{equation*}
                \scr{F}(\rho) + \frac{A^q}{q} \scr{I}_q(\mu) + \frac{1}{p A^p} \tau^p \scr{I}_q(\rho) \geq \scr{F}(\mu) + \tau^p \frac{\alpha}{p} \scr{I}_q(\rho)
            \end{equation*}
            Optimizing on $A$ gives the result. On the other hand, if we choose $A$ so that both $\scr{I}_q$ appears with the same coefficient, i.e. $A = \alpha^{-\frac{1}{p}}$, we obtain the second inequality.
            \item For the Fisher's bound, we use $H = h_q$ in the dissipation of Fisher \ref{prop: dissipation_information} to get
            \begin{equation*}
                \frac{1}{q} \scr{I}_q(\mu) \geq \frac{1}{q} \scr{I}_q(\rho) + \frac{1}{\tau} \int \omega(\nabla h^*_{\tau,q}(\nabla u_\rho)) \dd{\rho} = \frac{1}{q} (1 +  \tau^{p-1} \beta) \scr{I}_q(\rho)
            \end{equation*}
        \end{enumerate} 
    \end{proof}

    We can now deduce the following generalized Log-Sobolev inequality.

    \begin{theorem}[]
        Suppose that $V$ has modulus of convexity $\sigma = \frac{\alpha}{p} |\cdot|^p$, then $\eta$ satisfies the following Log-Sobolev inequality
        \begin{equation}
            \int_\Omega g e^g \dd{\eta} \leq \frac{1}{q \alpha^{q-1}} \int_\Omega |\nabla g|^q e^g \dd{\eta}
        \end{equation}
    \end{theorem}

    \begin{proof}
        We can choose any $\beta$ such that $\frac{\beta}{q} |\cdot|^p$ is a modulus of monotonicity (which exists since $2 \sigma$ is a modulus of monotonicity). Iterating the inequalities of proposition \ref{prop: p_power_dissipation} along the JKO scheme for a given $\tau$, we obtain the following inequalities: if we let $F_k := \scr{F}(\rho_k)$ and $I_k := \scr{I}_q(\rho_k)$ and working with the parameter $\theta = \beta \tau^{p-1}$ instead of $\tau$ This means that the following dissipation of entropy holds
        \begin{equation*}
            F_{k+1} + \frac{1}{q} \frac{1}{\alpha^{q-1}} I_k \geq F_k
        \end{equation*}
        and the dissipation of Fisher's $I_{k+1} \geq (1+\theta) I_k$, giving $I_k \leq (1+\theta)^{-k} I_0$. Now if we sum the inequality, using that $F_k \to 0$ as $k \to +\infty$ by \ref{prop: convergence_equlibrium} we obtain
        \begin{align*}
            F_0 &= \sum_{k=0}^\infty F_k - F_{k+1} \leq \frac{1}{q} \frac{1}{\alpha^{q-1}} \sum_{k=0}^\infty I_k \\
            &\leq \frac{1}{q} \frac{1}{\alpha^{q-1}} \left ( \sum_{k=0}^\infty \frac{1}{(1+\theta)^k} \right ) I_0 = \frac{1}{q} \frac{1}{\alpha^{q-1}} \left ( 1 + \frac{1}{\theta} \right ) I_0
        \end{align*}
        Passing to the limit as $\theta \to +\infty$, i.e. $\tau \to +\infty$, we deduce the inequality 
        \begin{equation*}
            \scr{F}(\rho_0) \leq \frac{1}{q \alpha^{q-1}} \scr{I}_q(\rho_0)
        \end{equation*}
        from which the generalized Log-Sobolev follows by change of variable.
    \end{proof}

    \begin{remark}
        This proof suggests that a Bakry-Emery's based proof of the above generalized $p$-Log-Sobolev inequality using the $p$-Laplace porous-medium flow 
        \begin{equation*}
            \partial_t \rho = \nabla \cdot (\rho |\nabla \log \rho|^{q-2} \nabla \log \rho) \propto \Delta_q \rho^\frac{1}{q-1} 
        \end{equation*} 
        might fail. Indeed, as when $\tau \to +\infty$, our dissipation bound explodes. On the other hand, for any $\tau > 0$ we can obtain a discrete version of this proof, and then pass to the limit to recover Cordero-Gangbo-Houdré result, which again suggests that our method is an interpolation between Bakry-Emery and Cordero-Gangbo-Houdré.
    \end{remark}

    Let's give a classical example of a potential which satisfies the above hypothesis:

    \begin{lemma}[Modulus of power-$p$] \label{lemma: modulus_power_p}
        Let $V_p = \frac{1}{p} |\cdot|^p$ with $p \geq 2$, then $V_p$ admits $2^{2-p} |\cdot|^p$ as a modulus of monotonicity, and $C(p) |\cdot|^p$ as modulus of convexity, with 
        \begin{equation}
            C(p) = \min_{t \in [-1,0]} \frac{1}{p} |t+1|^p - \frac{1}{p} |t|^p - |t|^{p-2} t \geq \frac{2^{2-p}}{p}
        \end{equation}
        Which is attained for the unique solution $t_p \in [-1,0]$ of the equation:
        \begin{equation}
            |t_p+1|^{p-2} (t_p+1) = |t_p|^{p-2} t_p + (p-1) |t_p|^{p-2}
        \end{equation}
    \end{lemma}

    \begin{proof}
        The proof is standard, and recalls some computations for pointwise vector inequalities used in the study of the p-laplacian. We include some element of proof for the sake of completeness. Replacing $x-y$ by a vector $v$ we want to show 
        \begin{align*}
            &(|x+v|^{p-2} (x+v) - |x|^{p-2} x) \cdot v \geq 2^{2-p} |v|^p \\
            &\frac{1}{p} |x+v|^p \geq \frac{1}{p} |x|^p + |x|^{p-2} x \cdot v + C(p) |v|^p
        \end{align*}
        In both case, we can assume $|v|=1$ by normalization (the case $v=0$ being trivial). Which gives
        \begin{align*}
            &(|x+v|^{p-2} (x+v) - |x|^{p-2} x) \cdot v \geq 2^{2-p} \\
            &\frac{1}{p} |x+v|^p - \frac{1}{p} |x|^p - |x|^{p-2} x \cdot v \geq C(p)
        \end{align*}
        We then minimize on $x$, a simple differentiation argument show that, at $x \notin \{ 0, - v\}$, the gradient (in $x$) of the two functions is of the form $\Phi(x) x + \Psi(v) v$ for some non-zero scalar functions $\Phi,\Psi$, hence optimality forces $x$ to be parallel to $v$, hence we can reduce to the $1$-d case an prove
        \begin{align*}
            &|t+1|^{p-2} (t+1) - |t|^{p-2} t \geq 2^{2-p} \\
            &\frac{1}{p} |t+1|^p - \frac{1}{p} |t|^p - |t|^{p-2} t \geq C(p)
        \end{align*}
        in the first case, again by critical point argument, one obtain that the minimum is attained at $t = -\frac{1}{2}$, which gives the first inequality. In the second case, again by critical point, optimality condition is
        \begin{equation*}
            |t+1|^{p-2} (t+1) = |t|^{p-2} t + (p-1) |t|^{p-2} 
        \end{equation*}
        and one can check that the optimal $t$ must be on $[-1,0]$, which gives the result (notice that we must have $C(p) > 0$ as $2^{2-p} \int_0^1 s^{-1} |s \cdot|^p = \frac{2^{2-p}}{p} |\cdot|^p$ is a modulus of convexity.
    \end{proof}

    \begin{remark}
        One can explicitly compute that $C(3) = \frac{1}{3}(2 - \sqrt{2})$ with $t_3 = \frac{1}{\sqrt{2}} - 1$, showing in general it is possible that the inequalities $\omega \geq 2 \sigma$ and $\sigma \geq \int_0^1 t^{-1} \omega(t \cdot) \dd{t}$ are strict.
    \end{remark}
    
\end{document}